\documentclass[12pt]{article}
\usepackage{amsmath, amsthm, amssymb, mathrsfs, fullpage, color,tikz}
\renewcommand{\footnote}[1]{}
\newtheorem{thm}{Theorem}
\newtheorem{lem}[thm]{Lemma}
\newtheorem{prop}[thm]{Proposition}

\title{The Asymptotic Behavior of Compositions of the Euler and Carmichael Functions}
\author{{\sc Vishaal Kapoor}\\ \texttt{vkapoor@google.com}}
\date{{\sc \today}}

\renewcommand{\phi}{\varphi}

\newcommand{\sumst}[1]{\sum_{\substack{#1}}}
\newcommand{\ds}{\displaystyle}
\newcommand{\all}{\lambda\lambda(n)}
\newcommand{\alf}{\lambda\phi(n)}
\newcommand{\afl}{\phi\lambda(n)}
\newcommand{\qqa}{Q_{q^{\alpha}}}
\newcommand{\oqqa}{\overline{Q}_{q^{\alpha}}}
\newcommand{\pqa}{P_{q^{\alpha}}}
\newcommand{\logratio}{\log \frac{\alf}{\all}}
\newcommand{\cP}{\mathcal{P}}
\newcommand{\nn}{\nonumber\\}
\newcommand{\bZ}{\mathbb{Z}}
\newcommand{\vphi}{\varphi}
\newcommand{\lcm}{\text{lcm}}

\renewcommand{\mod}[1]{{\ifmmode\text{\rm\ (mod~$#1$)}\else\discretionary{}{}{\hbox{ }}\rm(mod~$#1$)\fi}}

\begin{document}
\maketitle
\begin{abstract}
We compare the asymptotic behavior of $\lambda(\varphi(n))$ and $\lambda(\lambda(n))$ on a set of positive integers $n$ of asymptotic density 1, where $\lambda$ is Carmichael's $\lambda$-function and $\varphi$ is Euler's totient function.
We prove that $\log {\lambda(\vphi(n))}/{\lambda(\lambda(n))}$ has normal order $\log\log n \log \log \log n$.
\end{abstract}

\section{Introduction}
{\em Euler's totient function} $\vphi(n)$ is defined to be the cardinality of the multiplicative group modulo $n$, for any positive integer $n$.  {\em Carmichael's $\lambda$-function} \cite{C} denotes the cardinality of the largest cycle in the multiplicative group modulo $n$.  In other words, $\lambda(n)$ is the smallest positive integer $m$ such that $a^m \equiv 1 \mod n$ for all reduced residues $a \mod n$.  We notice that when the multiplicative group modulo $n$ is cyclic, namely when $n=1,2,4, p^a$ or $2p^a$ where $p$ is an odd prime and $a\geq 1,$ both $\phi(n)$ and $\lambda(n)$ are equal.  

One may compute $\vphi(n)$ with the aid of the Chinese remainder theorem by using the formula
\begin{align*}
\varphi(n) = |(\bZ/p_1^{a_1} \bZ)^{\times} |  \times \cdots \times |(\bZ/p_k^{a_k} \bZ)^\times| = p_1^{a_1-1}(p_1-1)\cdots p_k^{a_k-1} (p_k-1).
\end{align*}
where $n$ has the prime decomposition $n=p_1^{a_1}\cdots p_k^{a_k}$.   For Carmichael's function we note
\begin{align}\label{LPlambda1}
\lambda(p^a) = \left\{\begin{array}{ll}
p^{a-1}(p-1) & \text{if }p\geq 3 \text{ or } a \leq 2, \text{~and}\\
2^{a-2} &\text{if } p=2 \text { and } a\geq 3,
\end{array}\right.
\end{align}
together with 
\begin{align}\label{LPlambda2}
\lambda(n) = \lcm(\lambda(p_1^{a_1}),..., \lambda(p_k^{a_k})).
\end{align}

In what follows we introduce the following notation.  Given two functions $f(n)$ and $g(n)$, we will frequently drop the outer parentheses from the expression $f(g(n))$, instead writing the composition as $fg(n)$.  Additionally for $f(n)$ denoting $\lambda(n), \phi(n)$ or $\log(n)$, we define $f_1(n) = f(n)$ and $f_{k+1}(n) = f (f_{k}(n))$ for $k\geq 1.$   We will use the expression ``for almost all $n$'' to mean for $n$ in a set of positive integers of asymptotic density $1$, and the expression ``for almost all $n\leq x$'' to be analogous, but restricting $n\leq x$.  We recall that for arithmetic functions $f(n)$ and $g(n)$, we say $f(n)$ has normal order $g(n)$ if $f(n)$ is asymptotic to $g(n)$ for almost all $n$, or equivalently if $f(n) = (1+o(1))g(n)$ for almost all $n$.  

The theorem that we prove in this article is:
\begin{thm}\label{LPthm}
The normal order of $\displaystyle \log (\lambda\vphi(n)/\lambda\lambda(n))$ is $\log_2 n \log_3 n$.  
\end{thm}

More precisely, we show that for almost all $n \leq x$,
\begin{align}\label{LPeq100}
		\displaystyle \log \frac{\lambda\vphi(n)}{\lambda\lambda(n)} = \log_2 n \log_3 n + O(\psi(x)\log_2 x),
\end{align} where $\psi(x)$ is a function tending to infinity slower than $\log_3 x$.  We also show that the exceptional set of positive integers $n$ for equation \eqref{LPeq100} is of asymptotic density $O(x/\psi(x)).$ This work is part of the author's PhD thesis (see \cite{vkapoor}).

There has been extensive study on the asymptotic behavior of $\phi(n)$ and $\lambda(n)$ and their compositions.  In 1928, Schoenberg \cite{schoenberg} established that the quotient $n/\phi(n)$ has a continuous distribution function.  In other words: 
\begin{prop}
The limit
\begin{align*}
\Phi(t) = \lim_{N\rightarrow \infty} | \{ n \leq N : n/\vphi(n) \geq t\}|/N
\end{align*}
exists and is continuous for any real $t$.
\end{prop}
Recently Weingartner \cite{MR2317939} studied the asymptotic behavior of $\Phi(t)$ showing that as $t$ tends to infinity, $\log \Phi(t) = -\exp(t e^{-\gamma})(1+O(t^{-2}))$, where $\gamma = 0.5722...$ is Euler's constant.
 
 We mention that higher iterates of $\varphi(n)$ have been studied by Erd\H{o}s, Granville, Pomerance and Spiro in  \cite{MR1084181}.  They established:
\begin{prop}
The normal order of the $\vphi_k(n)/\vphi_{k+1}(n)$ is $k e^{\gamma} \log_3 n$, for $k\geq 1.$
\end{prop}
 
In 1955 Erd\H{o}s established the normal order of $\log (n/ \lambda(n))$ in \cite{MR0079031}.  This result was refined by Erd\H{o}s, Pomerance, and Schmutz in \cite{MR1121092} where they proved the following result.
\begin{prop}
For almost all $n\leq x$, 
\[
\log \frac{n}{\lambda(n)} = \log_2 n (\log_3 n + A + O((\log_3 n)^{-1 + \varepsilon}),
\]
where
\[
A = -1 + \sum_{q \textrm{ prime}} \frac{q}{(q-1)^2} = .2269688...,
\]
and $\varepsilon >0$ is fixed but arbitrarily small.

\end{prop}
The author is undertaking the analysis of Theorem 1 to obtain a more accurate asymptotic formula of a form more closely resembling the previous proposition. 

Martin and Pomerance subsequently considered the question of understanding the behavior of $\all$.  In \cite{MP} they proved
\begin{prop}\label{LPPropMP}
For almost all $n$, 
\begin{align} \log\frac n{\lambda\lambda(n)} = (1+o(1))(\log_2 n )^2 \log_3 n.\label{LPeq50}
\end{align}
\end{prop}
Recently Harland \cite{harland} proved a conjecture of Martin and Pomerance concerning the behavior of the higher iterates of $\lambda(n)$:
\begin{prop}
For each $k\geq 1$,
\begin{align*}
		\log \frac{n}{\lambda_k(n)} = \bigg(\frac{1}{(k-1)!}+o(1)\bigg)(\log_2 n)^k \log_3 n,
\end{align*}
for almost all $n$.
\end{prop}
Banks, Luca, Saidak, and St\u{a}nic\u{a} \cite{MR2187587} studied the the compositions of $\lambda$ and $\varphi$.  In particular, they studied set of $n$ on which $\alf = \afl.$  In their paper, they also established the following:
\begin{prop}\label{LPPropBLSS}
For almost all $n$, 
\begin{align}
\log\frac{n}{\afl} & = (1+o(1))\log_2 n \log_3 n, \text{ and} \label{LPeq52}\\
\log \frac{n}{\alf} & = (1+o(1)) (\log_2n)^2 \log_3 n.\label{LPeq51}
\end{align}
Consequently,
$\displaystyle \log \frac{\afl}{\alf}$ has normal order $(\log_2n)^2 \log_3 n.$
\end{prop}

The proof of Proposition \ref{LPPropBLSS} uses a simple clever argument that rests on the theorem of Martin and Pomerance.
 It is interesting to see what we may obtain trivially from Propositions \ref{LPPropMP} and \ref{LPPropBLSS}.  Subtracting \eqref{LPeq52} from \eqref{LPeq50} gives an asymptotic formula for the comparison between $\afl$ and $\all$,
\begin{align*}
\log \frac{\afl}{\all} \sim (\log_2 n)^2 \log_3 n,
\end{align*}
for almost all $n$.
However, if we subtract \eqref{LPeq51} from \eqref{LPeq50}, the main terms cancel and we are left with
\begin{align*}
\log \frac{\lambda\vphi(n)}{\lambda\lambda(n)} = o((\log_2 n)^2 \log_3 n),
\end{align*}
for almost all $n$.  This relation is interesting because it leads one to seek a more accurate asymptotic formula. This more accurate result is the content of Theorem \ref{LPthm}.  


\footnote{Personal note: What about comparing $\varphi\varphi$ to $\varphi\lambda$?  If this is easy, why not state it?  Don't we want to state how Theorem 1 is so great because it completes the picture of single compositions?}

\section{Notation and Useful Results}
Let $a,n\in\mathbb{Z}$.  Then the Brun-Titchmarsh inequality is the asymptotic relationship that
\begin{align}\label{LPeq40}
		\pi(z; n, a) \ll \frac{z}{\phi(n) \log (z/n)} \qquad (z>n),
\end{align}
where $\pi(z; n,a)$ is the number of primes congruent to $a\mod n$ up to $z.$  We will be primarily concerned with implications of the Brun-Titchmarsh inequality in the case that $a=1$.  For convenience, define $\mathcal{P}_n$ to be the set of primes congruent to $1 \mod n$, and for a given integer $m$, define the greatest common divisor of $m$ and $\mathcal{P}_n$, denoted $(m,\mathcal{P}_n)$, to be the product of the primes congruent to $1\mod n$ that divides $m$, or $1$ if none exist.  We will frequently use the following weaker form of \eqref{LPeq40} without mention.
\begin{lem}[A Brun-Titchmarsh Inequality]\label{LPlemma:bt}
For all $z>e^e$, 
\begin{align}\label{LPeq41}
\sumst{p\leq z \\ p \in \mathcal{P}_n} \frac1p \ll \frac{\log \log z }{\phi(n)}.
\end{align}
\end{lem}
One may obtain \eqref{LPeq41} from \eqref{LPeq40} by partial summation.
We will also use the following prime estimates stated in \cite{MP}.
\begin{lem}\label{LPlemma:primesums}
Let $z>e$.  Then we have the following: 
\begin{align*}
&\sum_{p \leq z} \log p \ll z, \qquad \sum_{p\leq z}\frac{\log p}{p} \ll \log z,\qquad \sum_{p\leq z} \frac{\log^2 p}{p} \ll \log^2 z, \\
&\sum_{p > z}\frac{\log p}{p^2} \ll \frac 1z,\text{ and} \qquad \sum_{p>z}\frac{1}{p^2} \ll \frac{1}{z\log z},
\end{align*}
\end{lem}
These estimates follow via partial summation applied to Mertens' estimate $ M(z) = \sum_{p\leq z} (\log p)/p = \log z + O(1).$  We illustrate the derivation of the first tail estimate.  One writes the Riemann-Steltjies integral
\begin{align*}
\sum_{p > z}(\log p)/p^2&= \int_{z}^{\infty} 1/t~\text{d}M(t) = M(t)/t\bigg|^{\infty}_z + \int_{z}^{\infty} M(t)/t^2 ~\text{d}t\\
	&=(\log z)/z+ O(1/z) + \int_{z}^{\infty} (\log t)/t^2+ O(1/t^2) ~\text{d}t\\
	&\ll 1/z^2,
\end{align*}
as required.

We remind the reader that we will be writing the composition of two arithmetic functions $f(n)$ and $g(n)$ as $fg(n)$, and subscripts will be used with functions to indicate the number of times a function will be composed with itself (ie $\log_2 n = \log\log n$).  The multiplicity to which a prime $q$ divides $n$ is denoted by $\nu_q(n).$  In what follows, the variables $p,q,r$ will be reserved for primes.  {\em Throughout, we denote $y=y(x) = \log_2 x.$}  The function $\psi(x)$ denotes a function tending to infinity, but slower than $\log y$.  
When we use the expression ``for almost all $n\leq x$'', we will mean for all positive integers $n\leq x$ except those in an exceptional set of asymptotic density $O(x/\psi(x))$. 
We will make use of two parameters $Y=Y(x)$ and $Z=Z(x)$ in the course of the proof of Theorem \ref{LPthm} which we now define as
\begin{align*}
Y &= 3cy, \text{ and}\\
Z &= y^2,
\end{align*}
where $c$ is the implicit constant appearing in the Brun-Titchmarsh theorem \eqref{LPeq40} and \eqref{LPeq41}. 
\section{The Proof of Theorem \ref{LPthm}}
We intend to establish an asymptotic formula for 
\begin{align}
\ds \logratio = \sum_{q}( \nu_q(\alf) - \nu_q(\all))\log q,\label{LPeq34}
\end{align}

valid for $n$ in a set of natural density 1.  We will consider the ``large'' $q$ and the ``small'' $q$ separately.  The cut-off for this distinction is the parameter $Y$ giving the cases $q>Y$ and $q\leq Y$, respectively.  \footnote{Don't neglect to deal with the prime q=2}  

For $q > Y$, it will be unusual for $\nu_q(\alf)$ to be strictly larger than $\nu_q(\all)$ and so the contribution in \eqref{LPeq34} from large $q$ will be negligible.  
We bound the sum in \eqref{LPeq34} by the two cases,
\begin{align}
\sum_{q > Y}( \nu_q(\alf) - \nu_q(\all))\log q
	&  \leq \sumst{q > Y \\ \nu_q(\alf) \geq 2} \nu_q(\alf)\log q \nn
	&	\qquad + \sumst{q > Y \\ \nu_q(\alf) = 1} (\nu_q(\alf) - \nu_q(\all))\log q.
\end{align}
We prove the two bounds:
\begin{prop}\label{LPprop1}
For almost all $n\leq x$,
\begin{align*}
\sum_{\substack{q > Y \\ \nu_q(\alf) = 1}} (\nu_q(\alf) - \nu_q(\all))\log q \ll y\psi(x),
\end{align*}
\end{prop}
and
\begin{prop}\label{LPprop2}
For almost all $n\leq x$,
 \begin{align}
\sum_{\substack{q > Y \\ \nu_q(\alf) \geq 2}} \nu_q(\alf)\log q \ll y\psi(x).
\end{align}
\end{prop}
Combining Propositions \ref{LPprop1} and \ref{LPprop2} gives the upper bound we seek:
\begin{prop}\label{LPprop3}
For almost all $n\leq x$,
\begin{align}
\sum_{q > Y}( \nu_q(\alf) - \nu_q(\all))\log q \ll y\psi(x).
\end{align}
\footnote{Personal note: Note this requires $Y>2cy$ where $c$ is the constant in the Brun-Titchmarsh theorem}
\end{prop}

We now consider with those primes $q\leq Y.$  It will turn out that the main term comes from the quantity $\sum_{q \leq Y} \nu_q(\alf)$ with the sum $\sum_{q \leq Y}\nu_q(\all)$ sufficiently small.
\begin{prop}\label{LPprop4}
For almost all $n\leq x$,
\begin{align*}
\sum_{q \leq Y} \nu_q(\all) \log q \ll y\psi(x).
\end{align*}
\end{prop}
We are left with the final piece of establishing the asymptotic behavior of  $\sum_{q \leq Y} \nu_q(\alf)$.  This will involve a case-by-case analysis of the various ways that $q$ can divide $\alf$ with multiplicity.  Two functions $g(n)$ and $h(n)$ arise from this analysis:
\begin{align*}
g(n) 	& = \sum_{q \leq Y}\sum_{\substack{\alpha \geq 1 \\ q^{\alpha+1} | \phi(n)}} \log q,\\
h(n) & = \sum_{q \leq Y} \sum_{\substack{\alpha \geq 1 \\ \omega(n, Q_{q^{\alpha}}) > 0}} \log q,\text{~and}\\
Q_{q^{\alpha}} &= \{ r\leq x : \exists p \in P_{q^{\alpha}}~\text{st}~ r\in P_p\}.
\end{align*}
We will show that $g(n)$ is a good approximation to $\sum_{q \leq Y} \nu_q(\alf)$.  To deal with $g(n)$, we will choose a suitably close additive function to approximate $g(n)$ and employ the Tur\'{a}n-Kubilius inequality to find the normal order of $g(n)$.
\begin{prop}\label{LPprop6}
For almost all $n\leq x,$
\begin{align*}
g(n) = y \log y + O(y).
\end{align*}
\end{prop}
\begin{prop}\label{LPprop7}
For almost all $n\leq x$,
\begin{align*}
h(n) \ll \psi(x)y.
\end{align*}
\end{prop}
We will combine these propositions to show
\begin{prop}\label{LPprop8}
\begin{align*}
\sum_{q\leq Y}( \nu_q(\alf) - \nu_q(\all))\log q = y \log y + O(\psi(x) y).
\end{align*}
\end{prop}

Summing the results from Propositions \ref{LPprop3} and \ref{LPprop8} gives 
\begin{align*}
	\sum_{q}( \nu_q(\alf) - \nu_q(\all))\log q = y \log y + O(\psi(x) y),
\end{align*}
which proves Theorem 1.
In the following two sections, we will establish all of the propositions of this section except proposition \ref{LPprop3} which we have established.
%
%
\section{Large Primes $q > Y$}
In this section we prove Propositions \ref{LPprop1} and \ref{LPprop2}.  In order to proceed, we must first understand the different ways in which prime powers can divide $\all$ and $\alf$.  We assume $Y\geq 2$ so all primes $q$ under consideration are odd. 

From the definition $\lambda(n)$ (see \eqref{LPlambda1} and \eqref{LPlambda2})\footnote{Question: Should I insert those equations here in the format ``Recall \dots''?}, one sees that $\all$ has $q$ as a prime divisor if $q^2$ divides $\lambda(n)$ or if $n$ is divisible by some prime in $\cP_q.$  We emphasize that these conditions are not exclusive.  We may expand these conditions in turn.  If $q^2|\lambda(n)$, then the higher power $q^3$ divides $n$, or a prime in $\cP_{q^2}$ divides $n$; while if some prime $p\in \cP_q$ divides $\lambda(n)$, then $p^2|n$, or $(n, \cP_p)>1$.
We summarize these cases in the tree diagram below.
\\
\newcommand{\st}{\text{ st }}
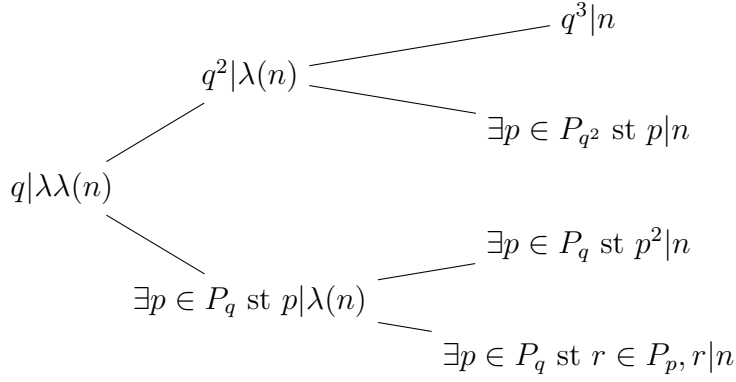
\begin{tikzpicture}
\path (0,2.25) node (v0) {$q|\all$};
\path (2.5, 3.75) node (v1) {$q^2 | \lambda(n)$};
\path (2.5,.75) node (v2) {$\exists p \in P_{q} \st p | \lambda(n)$};
\path (7, 4.5) node (v3) {$q^3 | n$};
\path (7, 3) node (v4) {$\exists p \in P_{q^2} \st p|n$};
\path (7, 1.5) node (v5) {$\exists p \in P_{q} \st p^2 | n$};
\path (7, 0) node (v6) {$\exists p \in P_{q}\st r \in P_p, r|n$};
\draw (v0) -- (v1) -- (v3);  
\draw (v0) -- (v2) -- (v5);
\draw (v1) -- (v4);
\draw (v2) -- (v6);
\end{tikzpicture}
%
%

We proceed with a similar analysis on the ways that $q$ can be a divisor of $\alf.$  We saw that either $q^2$ or some prime in $\cP_q$ must divide the argument $\phi(n)$ of $\alf$.  If two copies of $q$ divide $\phi(n)$, then their presence can come from the cube $q^3$ dividing $n$, two distinct primes dividing $n$ with each prime in $\cP_q$ contributing one factor of $q$, both $q^2|n$ and a prime $p\in\cP_q$ dividing $n$, or a single prime in $\cP_{q^2}$ dividing $n$.  In the other case, if a prime $p\in \cP_q$ divides $\phi(n)$,  then $p^2|n$ or $(n, \cP_p) > 1$.

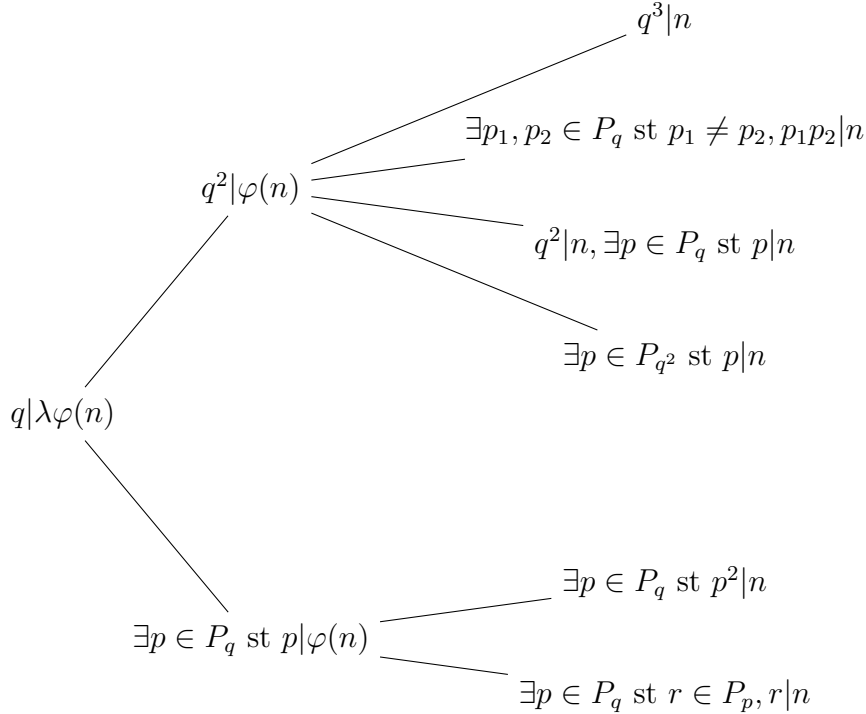
\begin{tikzpicture}
\path (0,3.75) node (v0) {$q|\alf$};
\path (2.5,6.75) node (v1) {$q^2 | \phi(n)$};
\path (2.5,.75) node (v2) {$\exists p\in P_{q} \st p|\phi(n)$};
\path (8,9) node (v3) {$q^3 | n$};
\path (8,7.5) node (v4) {$\exists p_1,p_2\in P_q \st p_1\neq p_2, p_1 p_2 |n$};
\path (8,6) node (v5) {$q^2 | n, \exists p\in P_q \st p|n$};
\path (8,4.5) node (v6) {$\exists p \in P_{q^2} \st p|n$};
\path (8,1.5) node (v7) {$\exists p\in P_{q}\st p^2 |n$};
\path (8,0) node (v8) {$\exists p\in P_{q}\st r \in P_p, r|n$};
\draw (v0)--(v1)--(v3);
\draw (v0)--(v2)--(v7);
\draw (v1)--(v4);
\draw (v1)--(v5);
\draw (v1)--(v6);
\draw (v2)--(v8);
\end{tikzpicture}

%
%

Now we turn to the proof of Proposition \ref{LPprop1}.

\begin{proof}[Proof of Proposition \ref{LPprop1}]
One sees from the above analysis that $q|\alf$ whenever $q|\all$, so the only way $(\nu_q(\alf) - \nu_q(\all))$ can be nonzero is if $q|\alf$ and $q\nmid\all$.
Moreover, there are only two ways that $q$ can divide $\alf$ but not $\all$; namely, two distinct primes $p_1,p_2 \in P_q$ could divide $n$, or both $q^2$ and a single prime $p\in P_q$ could divide $n$.  Thus
\begin{align*}
\frac1x\sum_{n\leq x}\sum_{\substack{q > Y \\ \nu_q(\alf) = 1}} (\nu_q(\alf) - \nu_q(\all))\log q 
	& \leq \frac1x \sum_{q > Y}\sum_{\substack{p_1,p_2 \in P_q \\ p_1 p_2 | n \\ n \leq x}} \log q  + \frac1x \sum_{q > Y} \sum_{\substack{n \leq x \\ p \in P_q \\ pq^2 | n}} \log q \\
	& \ll \frac1x \sum_{q > Y} \bigg(\frac{xy^2}{q^2} + \frac{xy}{q^3} \bigg)\log q\\
	& \ll y^2/Y,
\end{align*}
where we used Lemmata \ref{LPlemma:bt} and \ref{LPlemma:primesums}.  Plugging in $Y=3cy$ the upper bound is $\ll y$.  We deduce that for almost all $n\leq x$,
\begin{align*}
\sum_{\substack{q > Y \\ \nu_q(\alf) = 1}} (\nu_q(\alf) - \nu_q(\all))\log q \ll y\psi(x).
\end{align*}
\end{proof}
Now we would like to show that
 \begin{align}
\sum_{\substack{q > Y \\ \nu_q(\alf) \geq 2}} \nu_q(\alf)\log q \ll y^2\psi(x)/Y \label{LPeq14}
\end{align}
holds normally.  
\begin{proof}[Proof of Proposition \ref{LPprop2}]
Define $S_q = S_q(x) = \{ n\leq x : q^2 | n \text{ or } p|n \text{ for some } p\in P_{q^2} \}$ and $S = \cup_{q > Y} S_q.$  A simple estimate shows that the cardinality of $S$ is $O(xy/(Y\log Y))$.  We will choose $Y$ to be of asymptotic order $\gg y$, thus the number of elements in $S$ is $O(x/\psi(x))$.  As we are interested in a normality result, we may safely ignore the positive integers in $S$.  Consequently, to establish \eqref{LPeq14} for almost all $n$, it suffices to establish the mean value estimate
\begin{align}
\frac 1x \sum_{\substack{n\leq x \\ n \not\in S}} \sum_{\substack{q > Y \\ \nu_q(\alf) \geq 2}} \nu_q(\alf)\log q \ll y^2/Y.
\end{align}
To this end we write
\begin{align*}
\frac 1x \sum_{\substack{n\leq x\\n\not\in S}} \sum_{\substack{q > Y \\ \nu_q(\alf) \geq 2}} \nu_q(\alf)\log q &
	\leq \frac2x \sum_{\substack{q > Y \\ \alpha \geq 2}} \sum_{\substack{n\leq x \\n\not\in S\\ q^{\alpha} | \alf}}\log q \\
	&\leq \frac2x \sum_{\substack{q > Y \\ \alpha \geq 2}} \bigg(
	\sum_{\substack{n \leq x\\ p \in P_{q^{\alpha}} \\ p | \phi(n)}}+ \sum_{\substack{n \leq x \\ n\not\in S\\q^{\alpha+1} | \phi(n)}}\bigg) \log q.
\end{align*}
In order for the prime $p$ to be a divisor of $\phi(n)$, one of: $p^2$ divides $n$, or $r\in P_p$ and $r$ divides $n$ for some prime $r$ must occur.  Thus,
\begin{align}
\sum_{\substack{n \leq x \\ p \in P_{q^{\alpha}} \\ p | \phi(n)}} 1 = \sum_{\substack{p \leq x \\ p \in P_{q^{\alpha}}}} \sum_{\substack{n\leq x \\ p | \phi(n)}} 1 \ll 
\sum_{\substack{p \leq x \\ p \in P_{q^{\alpha}}}} \bigg(\frac x {p^2} +  \sum_{\substack{r\leq x \\ r\in P_p}}  \frac{x}{r}\bigg)  \ll \sum_{p > q^{\alpha}} \frac{x}{p^2} + \sum_{\substack{p \leq x\\ p\in P_{q^{\alpha}}}} \frac{xy}{p} \ll \frac{x}{\alpha q^{\alpha} \log q} + \frac{x y^2}{q^{\alpha}}.
\end{align}

Summing over $q>Y$ and $\alpha \geq 2$ and weighting by $\log q$ we have the asymptotic upper bound
\begin{align*}
\frac1x \sum_{\substack{q > Y \\ \alpha \geq 2}}\sum_{\substack{n \leq x \\ p \in P_{q^{\alpha}} \\ p | \phi(n)}} \log q \ll  y^2 /Y.
\end{align*}
Now we would like to establish
\begin{align*}
\frac 1x \sum_{\substack{q > Y \\ \alpha \geq 2}}  \sum_{\substack{n \leq x \\ n\not\in S\\q^{\alpha+1} | \phi(n)}} \log q\ll y^2/Y.
\end{align*}
We note that the contribution of prime powers of $q$ dividing $\phi(n)$ for $n\not\in S$ can only come from distinct primes in $P_q$ dividing n.  We then have
\begin{align}\label{LPeq15}
\sum_{\substack{n \leq x \\ n\not\in S\\q^{\alpha+1} | \phi(n)}}1 \ll \frac1{(\alpha+1)!} \sum_{p_1,...,p_{\alpha+1}\in P_q}\sum_{p_1\cdots p_{\alpha+1}|n \leq x} 1\ll \frac{x(cy)^{\alpha+1}}{(\alpha+1)!q^{\alpha+1}},
\end{align}
where we intentionally omit the condition that the primes $p_i \in P_q$ are distinct and where $c$ is the constant appearing in the Brun-Titchmarsh theorem.  
As $Y \geq 2cy$ we have $cy/q \leq 1/2$.  Thus summing the LHS of \eqref{LPeq15} over $\alpha \geq 2$ and $q > Y$ and weighting by $\log q$ gives
\begin{align}\label{LPeq16}
\sum_{q > Y} \sum_{\alpha \geq 3} \frac{xc^{\alpha}y^{\alpha}}{\alpha!q^{\alpha}} \log q 
	& \leq x c^2y^2 \sum_{\alpha \geq 1} \frac1{\alpha! 2^\alpha}\sum_{q > Y} \frac{\log q}{q^2}\ll xy^2/Y
\end{align}
 as required.
\end{proof}

\section{Small primes $q\leq Y$}\label{LPsmallprimes}
In this section we will be concerned with estimates for small primes; namely, we will prove Propositions \ref{LPprop4}, \ref{LPprop6}, \ref{LPprop7} and \ref{LPprop8}.  The main term in our asymptotic formula will come from Proposition \ref{LPprop6} which concerns the sum
\begin{align}\label{LPeq5}
\sum_{q \leq Y} \nu_q(\alf) \log q.
\end{align}
The remaining two Propositions provide us with error terms.\footnote{Personal note: Give the reader a roadmap for what you'll do in this section.  The relevant numbers are Propositions \ref{LPprop4}, \ref{LPprop6} (on $g(n)$), \ref{LPprop7} (on $h(n)$) and \ref{LPprop8}\footnote{Personal note: If these proposition numbers are finalized I will change this to 9-12} which combines them.  \ref{LPprop8} requires going through the original prop7a.}

We restate a Lemma 11 from \cite{MP} which we will use:
\begin{lem}\label{LPlem2}
For a power of a prime $q^a$, the number of positive integers $n\leq x$ with $q^a$ dividing $\all$ is $O(xy^2/q^a).$

\end{lem}
\footnote{Personal Note: Make sure you use $\alpha$ and not $a$ by accident.}
\begin{proof}[Proof of Proposition \ref{LPprop4}]
We break the summation up into two parts depending on the size of $q^\alpha,$
\begin{align*}
\sum_{q \leq Y} \nu_q(\all) \log q
	&= \sum_{q \leq Y} \log q \sumst{\alpha \geq 1 \\ q^\alpha | \all}1\\
	& \ll \sum_{q \leq Y} \log q \sumst{\alpha \geq 1 \\ q^\alpha \leq Z}1 + \sum_{q\leq Y} \log q\sumst{\alpha \geq 1 \\ q^\alpha >Z \\ q^\alpha | \all}1.
\end{align*}
We may bound the first sum as
\begin{align*}
\sum_{q \leq Y} \log q \sumst{\alpha \geq 1 \\ q^\alpha \leq Z}1 \ll Y \log Z/\log Y.
\end{align*}
We use an average estimate to bound the second sum.  Note
\begin{align}
\frac1x \sum_{n\leq x} \sum_{q\leq Y}\log q \sumst{\alpha\geq1 \\ q^\alpha > Z\\q^\alpha |\all} 1
	= \frac1x \sum_{q \leq Y} \log q \sumst{\alpha\geq1 \\ q^\alpha > Z} \sumst{n\leq x \\ q^\alpha | \all} 1.\label{LPeq55}
\end{align}
From Lemma \ref{LPlem2}, we see \eqref{LPeq55} is 
\begin{align*}
	\ll \frac1x \sum_{q\leq Y} \log q \sumst{\alpha\geq1 \\ q^\alpha > Z} \frac{x y^2}{q^\alpha}
	\ll\sum_{q\leq Y}\frac{y^2 \log q}{Z}
	\ll \frac{y^2 Y}{Z}.
\end{align*}
Therefore
\begin{align*}
\sum_{q\leq Y}\log q \sumst{\alpha\geq1 \\ q^\alpha > Z\\q^\alpha |\all} 1 \ll y^2 Y\psi(x) /Z,
\end{align*}
for almost all $n\leq x.$  Combining our upper bounds gives 
\begin{align*}
\sum_{q \leq Y}\nu_q(\all) \log q \ll (Y \log Z/\log Y + y^2 Y/Z) \psi(x),
\end{align*}
for almost all $n\leq x.$  Substituting $Y=3cy$ and $Z=y^2$ gives the theorem.
\end{proof}

%
%

Recall $q^{\alpha}$ divides $\alf$ if one of
\footnote{Personal note: In the case $q^{\alpha}=2^{\alpha}$, when $\alpha\geq 3$, we have the stronger first case $q^{\alpha+2}|\phi(n)$.}
\begin{itemize}
\item $q^{\alpha+1}|\phi(n)$
\item $q^{\alpha} | p-1, p | r-1, r | n$
\item $q^{\alpha} | p-1, p^2 | n$
\end{itemize}
occurs.  Note that these conditions are not mutually exclusive.  
We write \eqref{LPeq5} as
\begin{align*}
\sum_{q \leq Y} \nu_q(\alf) \log q  = g(n) + O\bigg(h(n) + \sum_{q \leq Y} \sum_{\substack{p \in P_{q^{\alpha}} \\ p^2 | n}} \log q\bigg),
\end{align*} 
where
\begin{align*}
g(n) 	& = \sum_{q \leq Y}\sum_{\substack{\alpha \geq 1 \\ q^{\alpha+1} | \phi(n)}} \log q,\\
h(n) & = \sum_{q \leq Y} \sum_{\substack{\alpha \geq 1 \\ \omega(n, Q_{q^{\alpha}}) > 0}} \log q, \textrm{~and}\\
Q_{q^{\alpha}} &= \{ r\leq x : \exists p \in P_{q^{\alpha}}~\text{st}~ r\in P_p\}.
\end{align*}
Thus, for almost all $n\leq x$,
\begin{align}
\sum_{q \leq Y} \nu_q(\alf) \log q  = g(n) + O(h(n) + \psi(x)\log_2 Y).\label{LPeq60}
\end{align}

In the next two sections, we prove Propositions \ref{LPprop6} and \ref{LPprop7}.  We see that Proposition \ref{LPprop8} follows immediately by applying these two propositions to equation \eqref{LPeq60} giving
\begin{align*}
\sum_{q\leq Y} \nu_q(\alf) \log q = y\log y +O(y\psi(x))
\end{align*}
for almost all $n\leq x$, as required.

\subsection{Normal order of $g(n)$}
Our strategy is to approximate $g(n)$ from above and below by an additive arithmetic function, thus indirectly making $g(n)$ amenable to the Tur\'{a}n-Kubilius inequality.  To start, write $g(n)$ as 
\begin{align}
g(n) 
	&= \sum_{q \leq Y}\sum_{\substack{ \alpha \geq 1 \\ q^{\alpha +1} | \phi(n)}} \log q\nonumber\\
	&= \sum_{q \leq Y}(\nu_q(\phi(n))-1)  \log q\nonumber\\
	&= \sum_{q \leq Y}\sum_{p|n} \nu_q(p-1) \log q - Y(1+o(1)) + O\bigg(\sum_{q \leq Y}\nu_q(n) \log q\bigg),\label{LPeq59a}
\end{align}
where we used the double inequality
\begin{align*}
\sum_{p | n} \nu_q(p-1) \leq \nu_q(\phi(n)) \leq \sum_{p|n} \nu_q(p-1) + \nu_q(n).
\end{align*}

\footnote{Personal note:{\em See the comment in the section on the normal order of $h(n)$ that concerns second order terms.  If we wanted to find a more accurate asymptotic formula than one would have to look at $Y$ here.  In our current argument we are ignoring it (ie. absorbing it into the error term).}}

We will use the Tur\'{a}n-Kubilius inequality:
\begin{lem}[The Tur\'an-Kubilius Inequality]
\label{LPTKlem}
There exists an absolute constant $C$ such that for all additive functions $f(n)$ and all $x\geq 1$ the inequality
\begin{align}\label{LPTKineq}
\sum_{n\leq x} |f(n) - A(x)|^2 \leq C x B(x)^2
\end{align}
holds where
\begin{align*}
A(x) &= \sum_{p\leq x} f(p)/p\text{, and}\\
B(x)^2 &= \sum_{p^k \leq x} |f(p^k)|^2/p^k.
\end{align*}
\end{lem}
\begin{proof}[Proof of Proposition \ref{LPprop6}]
We will use Lemma \ref{LPTKlem} for the additive function  $g_0(n) = \sum_{q\leq Y}\sum_{p|n}\nu_q(p-1)\log q$.  Let $A(x)$ and $B(x)$ be the first and second moments: 
\begin{align*}
A(x) &= \sum_{r \leq x} g_0(r)/r, \text{ and}\\
B(x) &=  \sum_{r^k\leq x} g_0(r^k)^2/ r^{k}.
\end{align*}
Notice that $g_0(r^k) = g_0(r) =  \sum_{q\leq Y}\nu_q(r-1)\log q$ leading to 
\begin{align*}
A(x) = \sum_{r \leq x} \frac1r \sum_{q \leq Y} \sum_{p|r} \nu_q(p-1)\log q
	& = \sum_{q\leq Y} \log q \sum_{r \leq x} \frac{\nu_q(r-1)}{r}\\
	& = \sum_{q\leq Y} \log q \sum_{\alpha \geq 1}\sum_{\substack{r \leq x \\ r \in P_{q^{\alpha}}}} \frac 1r.
\end{align*}
We split the sum over $\alpha$ into
\begin{align*}
 \sum_{1\leq \alpha \leq w_q}\sum_{\substack{r \leq x \\ r \in P_{q^{\alpha}}}} \frac 1r
 + \sum_{ \alpha >  w_q }\sum_{\substack{r \leq x \\ r \in P_{q^{\alpha}}}} \frac 1r,
 \end{align*}
 with $w_q$ to be determined later.
The first we estimate with Page's theorem and the second we bound with the Brun-Titchmarsh bound
$$\sum_{\substack{r \leq x \\r \equiv 1 \mod d}} 1/r \ll y/\phi(d).$$ 

 \begin{align}\label{LPeq6}
 \sum_{\alpha=1}^\infty \frac{y}{\phi(q^{\alpha})} 
 +O\bigg(\frac{y}{q^{w_q}}+ w_q\bigg) = \frac{yq}{(q-1)^2} + O\bigg(\frac{y}{q^{w_q}} + w_q\bigg)
 \end{align}
 Note used the\footnote{Personal note: best possible} bound $1/q^{\lfloor w_q \rfloor+1} = O(1/q^{w_q})$.
Taking $w_q = \log y/\log q$ gives an error term of $O(w_q) = O(\log y/\log q)$. 
Summing \eqref{LPeq6} over $q\leq Y$ weighted by $\log q$ gives the asymptotic formula
\begin{align}\label{LPeqm1}
A(x) 
	&= y \sum_{q\leq Y} \frac{q \log q}{(q-1)^2} + O\bigg(\frac{Y\log y}{\log Y} + Y\bigg)\nn
	&= y \log Y + O\bigg(\frac{Y\log y}{\log Y} + Y\bigg).
\end{align}

Expanding the square, write the second moment $B(x)$ as
\begin{align*}
B(x) 
	&= \sum_{q_1,q_2 \leq Y} \log q_1 \log q_2 \sum_{r\leq x}\nu_{q_1}(r-1) \nu_{q_2}(r-1)\sum_{\substack{k\leq 1\\ r^k \leq x}} 1/r^k.
\end{align*}
Uniformly in primes $r$, $\sum_{k \geq 1} 1/r^k \ll 1/r$.  We may also express $\nu_{q_i}(r-1)$ ($i=1,2$) as 
\begin{align*}
\nu_{q_i}(r-1) = \sum_{\substack{\alpha_i \geq 1 \\ r \in P_{q_i^{\alpha_i}}}} 1,
\end{align*}
giving the expanded
\begin{align*}
B(x) 
	&\ll \sum_{q_1,q_2 \leq Y} \log q_1 \log q_2 \sum_{\alpha_1,\alpha_2 \geq 1}\sum_{\substack{r\leq x\\ r\in P_{q_1^{\alpha_1}} \cap P_{q_2^{\alpha_2}}}}\frac1r.
\end{align*}
We split the sum in $q_1,q_2$ into the two cases: $q_1=q_2$ and $q_1\neq q_2$.   For the $q_1,q_2$ with $q = q_1 = q_2$ we have
\begin{align}
\sum_{q \leq Y} (\log q)^2 \sum_{\alpha_1,\alpha_2\geq 1}\sum_{\substack{r \leq x\\ r \in P_{q^{\max(\alpha_1,\alpha_2)}}}} \frac1r 
	& = \sum_{q \leq Y} (\log q)^2 \sum_{\alpha \geq 1}\sum_{\substack{r \leq x\\ r \in P_{q^{\alpha}}}} \frac\alpha r \nonumber \\
	& \ll \sum_{q \leq Y} (\log q)^2 \sum_{\alpha \geq 1}\frac{\alpha y}{q^{\alpha}} \nonumber \\
	& \ll y \sum_{q \leq Y} \frac{(\log q)^2}{q} \nonumber \\
	& \ll y (\log Y)^2. \label{LPeq201}
\end{align}

If $q_1$ and $q_2$ are distinct then we have an upper bound (intentionally ignoring the condition that $q_1\neq q_2$ in the sum)
\begin{align}
\sum_{q_1,q_2 \leq Y} \log q_1 \log q_2 \sum_{\alpha_1,\alpha_2 \geq 1}\sum_{\substack{r\leq x\\ r\in P_{q_1^{\alpha_1}q_2^{\alpha_2}}}}\frac1r
	& \ll \sum_{q_1,q_2 \leq Y} \log q_1 \log q_2 \sum_{\alpha_1,\alpha_2 \geq 1}\frac{y}{q_1^{\alpha_1}q_2^{\alpha_2}}\nonumber\\
	&\ll y \sum_{q_1,q_2 \leq Y} \frac{\log q_1 \log q_2}{q_1q_2}\nonumber\\
	&\ll y(\log Y)^2.\label{LPeq202}
\end{align}
Combining \eqref{LPeq201} and \eqref{LPeq202} gives \
\begin{align}
B(x) \ll y (\log Y)^2.
\end{align}
Using Lemma \ref{LPTKlem} we may conclude
that
The statement of Lemma \ref{LPTKlem} gives us the equation
\begin{align}
\sum_{n\leq x} |g_0(n) - A(x)|^2 \leq C x B(x)^2.
\end{align}
Thus the set of $n\leq x$ on which $g_0(n)$ differs from $A(x)$ by more than $y$ is $O(x(\log Y)^2/y) = O(x/\psi(x))$.  

The mean value of $\sum_{q\leq Y} \nu_q(n) \log q$ for $n\leq x$ is
$	 \ll 1/x \sum_{q\leq Y} x \log q /q 
	 \ll \sum_{q \leq Y} \log q/q \sim \log Y, $ so $\sum_{q\leq Y} \nu_q(n) \log q \ll \log^2 Y$ for almost all $n\leq x$. Thus from \eqref{LPeq59a}, we see that for almost all $n\leq x$, 
\begin{align}
g(n) = y \log Y + O\bigg(\frac{Y\log y}{\log Y} + Y\bigg),
\end{align}
Substituting $Y=3cy$ gives the theorem.
\end{proof}
\subsection{Normal order of $h(n)$}  
\begin{proof}[Proof of Proposition \ref{LPprop7}]
In order to find an upper bound on a set of asymptotic density 1, we will compute the first moment of $h(n)$:
\begin{align*}
H(x)
	&:= \frac1x\sum_{n\leq x} h(n) = \frac1x\sum_{\substack{q \leq Y \\ \alpha \geq 1}} \sum_{\substack{n\leq x \\ \omega(n, Q_{q^{\alpha}}) > 0 }}\log q\\
	&= \frac1x \sum_{\substack{q^{\alpha} \leq Z\\ q \leq Y\\ \alpha \geq 1}} \sum_{\substack{n\leq x \\ \omega(n, Q_{q^{\alpha}}) > 0 }}\log q + \frac1x \sum_{\substack{q^{\alpha} > Z \\ q \leq Y}} \sum_{\substack{n\leq x \\ \omega(n, Q_{q^{\alpha}}) > 0 \\ \alpha \geq 1 }}\log q.
\end{align*}

We deal with the two sums in turn.
\paragraph{Small $q^{\alpha}$}
The first part is for small powers of $q$:
\begin{align}
		\frac1x \sum_{\substack{q^{\alpha} \leq Z\\ q \leq Y}} \sum_{\substack{n\leq x \\ \omega(n, Q_{q^{\alpha}}) > 0 }}\log q & \leq \frac1x \sum_{\substack{q^{\alpha} \leq Z\\ q \leq Y}}\log q \sum_{n\leq x} 1 \leq\sum_{\substack{q^{\alpha} \leq Z\\ q \leq Y}}\log q  = \frac{Y\log Z}{\log Y}.\label{LPeq63}
 \end{align}
\paragraph{Large $q^{\alpha}$} The second part is for large powers of $q$.  In this case we use a crude estimate that is sufficient for our needs:
 \begin{align}
 \frac1x \sum_{\substack{q^{\alpha} > Z \\ q \leq Y}} \sum_{\substack{n\leq x \\ \omega(n, Q_{q^{\alpha}}) > 0 }}\log q 
 	& \ll \frac 1x \sum_{\substack{q^{\alpha} > Z \\ q \leq Y}}\log q \sum_{r \in Q_{q^{\alpha}}} \sum_{\substack{ n\leq x \\ r | n}}1\nonumber\\
	& \ll \frac1x \sum_{\substack{q^{\alpha} > Z \\ q \leq Y}}\log q \sum_{r \in Q_{q^{\alpha}}} \frac{x}{r}\nonumber\\
	& \ll  \sum_{\substack{q^{\alpha} > Z \\ q \leq Y}}\log q \sum_{p \in P_{q^{\alpha}}} \sum_{r \in P_{p}}\frac{1}{r}\nonumber \\ 
	& \ll y^2 \sum_{\substack{q^{\alpha} > Z \\ q \leq Y}} \frac{\log q}{q^{\alpha}}.\label{LPeq62}
\end{align}
The RHS of \eqref{LPeq62} is less than $\sum_{q \leq Y} \sum_{\alpha > \log Z/\log q} \log q/ q^{\alpha} \leq 2 \sum_{q \leq Y}\log q/q^{\log Z/\log q} \ll Y/Z$, or alternatively $q^{\alpha} \geq Z$ and $\sum_{q\leq Y} \log q \sim Y$.

Thus
\begin{align}
\frac1x \sum_{\substack{q^{\alpha} > Z \\ q \leq Y}} \sum_{\substack{n\leq x \\ \omega(n, Q_{q^{\alpha}}) > 0 }}\log q 
=O (y^2 Y/Z).\label{LPeq64}
\end{align}
Summing \eqref{LPeq63} and \eqref{LPeq64} gives 
\begin{align*}
		H(x) \ll Y\log Z/\log Y + y^2Y/Z \ll y,
\end{align*}
where we substituted the values of $Y$ and $Z$.  Thus, for almost all $n\leq x$, 
\begin{align*}
h(n) \ll y\psi(x).
\end{align*}
\end{proof}

\paragraph{Acknowledgements} I would like to thank my supervisor, Greg Martin, for his support and encouragement during the preparation of this article.

\bibliographystyle{plain}
\bibliography{LambdaPhi}

\begin{thebibliography}{10}

\bibitem{MR2187587}
W.~D. Banks, F.~Luca, F.~Saidak, and P.~St{\u{a}}nic{\u{a}}.
\newblock Compositions with the {E}uler and {C}armichael functions.
\newblock {\em Abh. Math. Sem. Univ. Hamburg}, 75:215--244, 2005.

\bibitem{C}
R.~D. Carmichael.
\newblock On {C}omposite {N}umbers {$P$} {W}hich {S}atisfy the {F}ermat
  {C}ongruence {$a^{P-1} \equiv 1 \operatorname{mod} P$}.
\newblock {\em Amer. Math. Monthly}, 19(2):22--27, 1912.

\bibitem{MR0079031}
P.~Erd{\"o}s.
\newblock On pseudoprimes and {C}armichael numbers.
\newblock {\em Publ. Math. Debrecen}, 4:201--206, 1956.

\bibitem{MR1084181}
P.~Erd{\H{o}}s, A.~Granville, C.~Pomerance, and C.~Spiro.
\newblock On the normal behavior of the iterates of some arithmetic functions.
\newblock In {\em Analytic number theory ({A}llerton {P}ark, {IL}, 1989)},
  volume~85 of {\em Progr. Math.}, pages 165--204. Birkh\"auser Boston, Boston,
  MA, 1990.

\bibitem{MR1121092}
Paul Erd{\H{o}}s, Carl Pomerance, and Eric Schmutz.
\newblock Carmichael's lambda function.
\newblock {\em Acta Arith.}, 58(4):363--385, 1991.

\bibitem{harland}
Nick Harland.
\newblock The iterated carmichael lambda function.
\newblock preprint.

\bibitem{vkapoor}
Vishaal Kapoor.
\newblock {\em Asymptotic formulae for arithmetic functions}.
\newblock PhD thesis, The University of British Columbia, April 2011.

\bibitem{MP}
Greg Martin and Carl Pomerance.
\newblock The iterated {C}armichael {$\lambda$}-function and the number of
  cycles of the power generator.
\newblock {\em Acta Arith.}, 118(4):305--335, 2005.

\bibitem{schoenberg}
Isac Schoenberg.
\newblock \"{U}ber die asymptotische {V}erteilung reeller {Z}ahlen mod 1.
\newblock {\em Math. Z.}, 28(1):171--199, 1928.

\bibitem{MR2317939}
Andreas Weingartner.
\newblock The distribution functions of {$\sigma(n)/n$} and {$n/\phi(n)$}.
\newblock {\em Proc. Amer. Math. Soc.}, 135(9):2677--2681 (electronic), 2007.

\end{thebibliography}
\end{document}